\documentclass{amsart}
\usepackage{amscd,amssymb,amsmath,microtype}
\usepackage[hyphens]{url}
\usepackage{hyperref}
\usepackage{lmodern}
\usepackage{xcolor}
\usepackage{colortbl}
\usepackage[most]{tcolorbox}
\usepackage{booktabs}
\usepackage{longtable}
\usepackage{changepage}
\usepackage{caption}
\usepackage{listings}
\usepackage{amssymb}
\usepackage{amsmath}
\usepackage{xypic}
\usepackage{nicefrac}
\usepackage{tikz}
\usepackage{multirow}
\usepackage{hyperref}
\usepackage{subcaption} 
\usepackage{mathtools}
\usepackage[draft,multiuser,layout={margin,index}]{fixme}
\usepackage{pifont}

\fxsetup{theme=color}

\FXRegisterAuthor{s}{ans}{H} 
 
\FXRegisterAuthor{j}{anj}{J}

\definecolor{white}{rgb}{1,1,1}
\definecolor{mygreen}{rgb}{0,0.4,0}
\definecolor{light_gray}{rgb}{0.97,0.97,0.97}
\definecolor{mykey}{rgb}{0.117,0.403,0.713}
\definecolor{mycomment}{rgb}{0.58,0.28,0}

\tcbuselibrary{listings}
\newlength\inwd
\newcounter{ipythcntr}
\renewcommand{\theipythcntr}{\texttt{[\arabic{ipythcntr}]}}

\newtcblisting{pyin}[1][]{%
  sharp corners,
  enlarge left by=\inwd,
  width=\linewidth-\inwd,
  enhanced,
  boxrule=0pt,
  colback=light_gray,
  listing only,
  top=0pt,
  bottom=0pt,
  overlay={
    \node[
      anchor=north east,
      text width=\inwd,
      font=\ttfamily\color{mykey},
      inner ysep=2mm,
      inner xsep=0pt,
      outer sep=0pt
      ] 
      at (frame.north west)
      {\refstepcounter{ipythcntr}\label{#1}In \theipythcntr:};
  }
  listing engine=listing,
  listing options={ 
    mathescape,
    xleftmargin=-2em,
    aboveskip=1pt,
    belowskip=1pt,
    basicstyle=\ttfamily\linespread{0.5},
    language=Python,
    keywordstyle=\color{mykey},
    showstringspaces=false,
    stringstyle=\color{mygreen},
    commentstyle=\color{mycomment}\itshape
  },
}
\newtcblisting{pyprint}{
  breakable,
  sharp corners,
  enlarge left by=\inwd,
  width=\linewidth-\inwd,
  enhanced,
  boxrule=0pt,
  colback=white,
  listing only,
  top=0pt,
  bottom=0pt,
  overlay={
    \node[
      anchor=north east,
      text width=\inwd,
      font=\ttfamily\color{mykey},
      inner ysep=2mm,
      inner xsep=0pt,
      outer sep=0pt
      ] 
      at (frame.north west)
      {};
  }
  listing engine=listing,
  listing options={
      escapechar=\%,
      mathescape,
      aboveskip=1pt,
      belowskip=1pt,
      basicstyle=\ttfamily,
      keywordstyle=\color{mykey},
      showstringspaces=false,
      stringstyle=\color{mygreen}
    },
}
\newtcblisting{pyout}[1][\theipythcntr]{
  sharp corners,
  enlarge left by=\inwd,
  width=\linewidth-\inwd,
  enhanced,
  boxrule=0pt,
  colback=white,
  listing only,
  top=0pt,
  bottom=0pt,
  overlay={
    \node[
      anchor=north east,
      text width=\inwd,
      font=\footnotesize\ttfamily\color{mykey},
      inner ysep=2mm,
      inner xsep=0pt,
      outer sep=0pt
      ] 
      at (frame.north west)
      {\setcounter{ipythcntr}{\value{ipythcntr}}Out#1:};
  }
  listing engine=listing,
  listing options={
      aboveskip=1pt,
      belowskip=1pt,
      basicstyle=\footnotesize\ttfamily,
      language=Python,
      keywordstyle=\color{mykey},
      showstringspaces=false,
      stringstyle=\color{mygreen}
    },
}

\makeatletter
\newcommand*{\@rowstyle}{}

\newcommand*{\rowstyle}[1]{
  \gdef\@rowstyle{#1}%
  \@rowstyle\ignorespaces%
}

\newcolumntype{=}{
  >{\gdef\@rowstyle{}}%
}

\newcolumntype{+}{
  >{\@rowstyle}%
}

\makeatother

\newtheorem{theorem}{Theorem}[section]

\newtheorem{corollary}[theorem]{Corollary}
{\theoremstyle{remark}
\newtheorem{remark}[theorem]{Remark}

}

\theoremstyle{definition}

\newtheorem{definition}[theorem]{Definition}
\newtheorem{example}[theorem]{Example}

\newcommand{\cmark}{\ding{51}}%
\newcommand{\xmark}{\ding{55}}%

\renewcommand{\Vert}{\mathcal{V}}
\newcommand{\Hor}{\mathcal{H}}

\newcommand{\T}{{\mathcal{T}}}

\newcommand{\CC}{\mathbb{C}}
\renewcommand{\L}{\mathcal{L}}
\newcommand{\RR}{\mathbb{R}}

\newcommand{\QQ}{\mathbb{Q}}
\newcommand{\ZZ}{\mathbb{Z}}
\newcommand{\NN}{\mathbb{N}}
\newcommand{\PP}{\mathbb{P}}
\newcommand{\A}{\mathbb{A}}

\newcommand{\CO}{{\mathcal{O}}}

\newcommand{\X}{{\mathcal{X}}}

\newcommand{\bphi}{\bar \Phi}

\DeclareMathOperator{\DF}{DF}

\DeclareMathOperator{\GL}{GL}

\DeclareMathOperator{\wdiv}{Div}

\DeclareMathOperator{\supp}{supp}

\DeclareMathOperator{\conv}{conv}

\DeclareMathOperator{\vol}{vol}

\DeclareMathOperator{\Cox}{Cox}
\DeclareMathOperator{\Cl}{Cl}
\DeclareMathOperator{\Ric}{Ric}

\newcommand{\savefootnote}[2]{\footnote{\label{#1}#2}}
\newcommand{\repeatfootnote}[1]{\textsuperscript{\ref{#1}}}

\title[K\"ahler-Ricci solitons on Gorenstein del Pezzo surfaces]{On the classification of K\"ahler-Ricci solitons\\ on Gorenstein del Pezzo surfaces}

\author[J. Cable]{Jacob Cable}
\address{Jacob Cable\\ School of Mathematics, Faculty of Science and Engineering,
The University of Manchester,
Alan Turing Building, Oxford Road,
Manchester M13 9PL}
\email{\href{mailto:jacob.cable@manchester.ac.uk}{jacob.cable@manchester.ac.uk}}
\author[H. S\"u{\ss}]{Hendrik S\"u\ss}
\address{Hendrik S\"u\ss\\ School of Mathematics, Faculty of Science and Engineering
The University of Manchester,
Alan Turing Building, Oxford Road, 
Manchester M13 9PL, United Kingdom}
\email{\href{mailto:hendrik.suess@manchester.ac.uk}{hendrik.suess@manchester.ac.uk}}

\subjclass[2010]{32Q20 (Primary) 14L30, 14J45 (Secondary)}
\keywords{K-stability, Kahler-Ricci solitons, T-variety, torus action, Fano variety}

\begin{document}
\begin{abstract}
  We give a classification of all pairs $(X,\xi)$ of Gorenstein del Pezzo surfaces $X$ and vector fields $\xi$ which are K-stable in the sense of Berman-Witt-Nystr\"om and therefore are expected to admit a K\"ahler-Ricci solition. Moreover, we provide some new examples of Fano threefolds admitting a K\"ahler-Ricci soliton.
\end{abstract}

\maketitle
\section{Introduction}

By a Fano orbifold we mean a complex normal variety with only orbifold singularities and an ample anti-canonical class. It is called Gorenstein if the anti-canonical class is Cartier. In dimension $2$ those varieties are usually called Gorenstein del Pezzo surfaces. Let $X$ be a Fano orbifold and $\omega_g$ be the K\"ahler form of a K\"ahler metric $g$ on $X$. The form $\omega_g$ is called a \emph{K\"ahler-Ricci soliton} if there exists an orbifold holomorphic vector field $\xi$, such that
\[
\Ric(\omega_g)-\omega_g = L_\xi\omega_g
\]
holds, where $L_\xi\eta$ denotes the Lie derivative of a form $\eta$ with respect to $\xi$. This implies that $L_{\xi}\omega_g$ is real-valued and, hence, $L_{\Im \xi}\omega_g = 0$, where $\Im \xi$ denotes the imaginary part of $\xi$. Therefore, $\Im \xi$ generates a one-dimensional Hamiltonian torus action on $X$. In the following we will identify $\xi$ with this action. If $\xi=0$ the metric $g$ is called \emph{K\"ahler-Einstein}, else we speak of a \emph{non-trivial} K\"ahler-Ricci soliton.

By \cite{tian02} we know that such a soliton metric is unique if it exists. On the other hand, by \cite{berman2014complex} together with \cite{datar16} the existence of such a metric (at least in the smooth case)  correponds to an algebro-geometric stability condition, known as K-stability. The key objects involved in defining K-stability are test configurations: 
\begin{definition}
Let $(X,L)$ be a polarized projective variety.
  A \emph{test configuration} for $(X,L)$ is a $\CC^*$-equivariant flat family $\X$ over $\A^1$ equipped with a relatively ample equivariant $\QQ$-line bundle $\L$ such that 
\begin{enumerate}
\item The $\CC^*$-action $\lambda$ on $(\X,\L)$ lifts the standard $\CC^*$-action on $\A^1$;
\item The general fiber is isomorphic to $X$, with $\L$ restricting to $L$.
\end{enumerate}
A test configuration with $\X \cong X \times \A^1$ is called a \emph{trivial} or a \emph{product configuration}. A test configuration with normal special fiber $\X_0$ is called \emph{special}. 
Given an algebraic group $G$ with action on $X$, a test configuration is called \emph{$G$-equivariant} if the action extends to $(\X,\L)$ and commutes with the $\CC^*$-action of the test configuration.
\end{definition}

\begin{remark}
 By Hironaka's Lemma \cite[III.9.12]{MR0463157} the normality of the fibers induces the normality of the total space of the family. Hence, a special test configuration $(\X,\L)$ has normal total space $\X$.
\end{remark}


We will primarily be concentrating on the situation where $X$ is Fano and $L=\CO(- K_X)$. It follows that the special fiber $\X_0$ is $\QQ$-Fano. 
 We proceed to define the modified Donaldson-Futaki invariants as they appeared in \cite{berman2014complex}.

Let $\X_0$ be any $\QQ$-Fano variety equipped with the action of an algebraic torus $T'$, and let $\ell\in \NN$ be the smallest natural number such that $-\ell K_{\X_0}$ is Cartier. Let $M'$ and $N'$ be the lattices of characters and one-parameter subgroups of $T'$, respectively. We denote the associated $\RR$-vector spaces by $M'_\RR$ and $N'_\RR$, respectively. Fix an element $\xi\in N'_\RR$.
 Consider the canonical linearisation for $\L_0 = \CO(-\ell K_{\X_0})$ coming from the identification $\L_0 \cong (\bigwedge^{\dim X} \T_{\X_0})^{\otimes \ell}$. Then we set $l_k=\dim H^0(\X_0,\L_0^{\otimes k})$ and for every $v \in N'$
\[w_k(v)= \sum_{u \in M'} \langle u,v\rangle \cdot e^{\nicefrac{\langle u, \xi \rangle}{k}}\cdot \dim H^0(\X_0,\L_0^{\otimes k})_{u}.\]
Now, 
\begin{equation}
  \label{eq:3}
  F_{\X_0,\xi}(v):=-\lim_{k \to \infty} \frac{w_k(v)}{k\cdot l_k\cdot \ell},
\end{equation}
defines a linear form on $N_\RR'$, i.e. an element of $M_\RR'$. It is called the (modified) Futaki character.

Consider now a Fano variety $X$ with a (possibly trivial) torus $T$ acting on it. Fix  $\xi \in N_\RR$, where $N$ is the lattice of one-parameter subgroups of $T$. Let $\X_0$ be the central fiber of a special $T$-equivariant test configuration $\X$ for $(X,\CO(-K_X))$. Then $\X_0$ comes equipped with a $T'=T\times \CC^*$-action induced by the test configuration. Let $v\in N'$ denote the one-parameter subgroup corresponding to the $\CC^*$-action of the test configuration. Note that the inclusion $T\hookrightarrow T'$ induces an inclusion of one-parameter subgroups $N\hookrightarrow N'$.

\begin{definition}
 Given a special test configuration $(\X,\L)$ of $(X,\CO(-K_X))$ as above,
 its \emph{modified Donaldson-Futaki invariant} is defined as \[\DF_\xi(\X)=F_{\X_0,\xi}(v).\]
\end{definition}

\begin{definition}
\label{def:k-stability}
  Consider a Fano variety $X$ with action by a reductive group $G$ containing a maximal torus $T \subset G$ and $\xi \in N_\RR$. The pair $(X,\xi)$ is called \emph{equivariantly K-stable} if $\DF_\xi(\X) \geq 0$ for every $G$-equivariant special test configuration $\X$ as above and we have equality exactly in the case of product test configurations. If $\xi=0$ we say $X$ itself is \emph{equivariantly K-stable}.
\end{definition}

The following result by Datar and Sz\'ekelyhidi motivates the study of \emph{equivariant} K-stability:
\begin{theorem}[{\cite{datar16}}]
\label{thm:gabor}
For a smooth Fano $G$-variety $X$, the variety  $X$ admits a K\"ahler-Ricci soliton with respect to $\xi$ if and only if the pair $(X,\xi)$ is equivariantly K-stable.
\end{theorem}

By \cite[6.1.2]{zbMATH06546438} the Gorenstein del Pezzo surfaces of degree $\leq 4$ admitting a K\"ahler-Einstein metric are known to be exactly those which are either smooth or fit into one of the following combinations of degree and singularity type.
    \begin{description}
  \item[Degree $1$] $2D_4$ or a combination of $A_k$-singularities with $k \leq 7$,
  \item[Degree $2$] $2A_3$ or a combination of $A_1$ and $A_2$ singularities,
  \item[Degree $3$] $3A_2$ or $\ell A_1$ with $\ell \geq 1$,
  \item[Degree $4$] $2A_1$ or $4A_1$.
  \end{description}
Hence, it is natural to ask which of the remaining ones admit at least a K\"ahler-Ricci soliton. In this paper we approach this question by giving a complete classification of pairs $(X,\xi)$ of Gorenstein del Pezzo surfaces $X$ and vector fields $\xi$ as above which are equivariantly K-stable with respect to the torus action generated by $\xi$. 

\begin{theorem}
Among the Gorenstein del Pezzo surfaces with non-trivial holomorphic vector fields. Exactly the following combinations of degree/singularity type admit non-trivial equivariantly K-stable pairs $(X,\xi)$
  \label{thm:surface-solitons}
  \begin{description}
  \item[Degree $1$] $E_8$, $E_7A_1$, $E_6A_2$
  \item[Degree $2$] $A_5A_2$, $D_6A_1$, $E_7$, $E_6$, $D_5A_1$, $D_43A_1$, 
  \item[Degree $3$] $E_6$, $A_4A_1$, $D_5$, $D_4$, $A_32A_1$, 
  \item[Degree $4$] $D_5$, $D_4$, $A_4$, $A_3$,\\
    toric: $A_32A_1$, $A_22A_1$, 
  \item[Degree $5$] $A_4$, $A_3$, $A_1$,\\
    toric: $2A_1$, $A_2A_1$ 
  \item[Degree $6$] $A_1$,\\
    toric: $A_1$, $2A_1$, $A_2A_1$, 
  \item[Degree $7$] toric: \emph{smooth}, $A_1$,
  \item[Degree $8$] toric: \emph{smooth}, $A_1$.
  \end{description}
\end{theorem}
For the toric cases the existence of a K\"ahler-Ricci soliton is known by \cite{Shi2012}.  Note, that Theorem~\ref{thm:gabor} only considers the smooth case. We hope and expect that the methods from \cite{datar16} will work in our setting as well. However, at the moment we do not have the corresponding statement in the case of orbifolds. This prevents us from actually proving the existence of K\"ahler-Ricci solitons in the cases considered in Theorem~\ref{thm:surface-solitons}. 

On the other hand, the implication of K-stability by the existence of a K\"ahler-Ricci soliton holds also in the singular case by \cite[Theorem 1.5]{berman2014complex}. This allows us to rule out the existence of K\"ahler-Ricci solitons for the remaining Gorenstein del Pezzo surfaces with non-trivial vector fields. Hence, by using a classification of such surfaces from \cite{elaine} we obtain the following corollary.

\begin{corollary}
 \label{cor:surface-non-solitons}
  The following combinations of degree/singularity type do not admit a K\"ahler-Ricci soliton metric (neither non-trivial nor K\"ahler-Einstein).
  \begin{description}
  \item[Degree $3$] $A_5A_1$, $2A_2A_1$, $2A_2$
  \item[Degree $4$] $A_3A_1$, $3A_1$
  \item[Degree $5$] $A_24$,
  \item[Degree $6$] $A_2$.
  \end{description}
\end{corollary}

In \cite{tfano} certain smooth Fano threefolds with 2-torus action were considered and the paper \cite{is15} determined which of them admit a K\"ahler-Einstein metric via equivariant K-stability. Moreover, for some of the remaining ones the existence of a K\"ahler-Ricci soliton could be proven by the simple observation that in these cases there are no equivariant special test configurations beside the product ones. In this paper we consider some of the remaining cases and obtain the following theorem.

\begin{theorem}
  \label{thm:threefold-solitons}
  The Fano threefolds 2.30, 2.31, 3.18, 3.22, 3.23, 3.24, 4.8 from Mori and Mukai's classification \cite{mori-mukai} admit a non-trivial  K\"ahler-Ricci soliton.
\end{theorem}

The common feature of the surfaces and threefolds considered in Theorem~\ref{thm:surface-solitons} and Theorem~\ref{thm:threefold-solitons}, repectively, is the presence of an effective action of an algebraic torus of one dimension less than the variety it acts on. In the following we will call these varieties \emph{T-varieties of complexity $1$}. Note, that surfaces with non-trivial K\"ahler-Ricci soliton automatically fall into this class due to the torus action generated by vector field. However, for threefolds this is indeed an additional condition.

In Section~\ref{sec:t-vars} we review the combinatorial description of Fano T-varieties of complexity $1$ and their equivariant test configurations as it was developed in \cite{is15}.

In Section~\ref{sec:classification} we state the classification of Gorenstein del Pezzo surfaces from \cite{elaine} in terms of their combinatorial data and describe the computational methods, that we used to determine which of these surfaces can be complemented to a K-stable pair in the sense of Definition~\ref{def:k-stability}.

Finally, in Section~\ref{sec:threefolds} we apply the same methods to the threefolds from \cite{tfano} and \cite{is15} to obtain new examples of K\"ahler-Ricci solitons on Fano threefolds.

In an appendix we provide examples of the computer assisted calculations, which we used to obtain our results. The complete computations are available in the ancillary files \cite{sage-code}.

\section{Combinatorial description of T-varieties of complexity $1$}
\label{sec:t-vars}
We fix an algebraic torus $T \cong (\CC^*)^n$. We denote its character lattice by $M$ and the dual lattice of co-characters or one-parameter subgroups by $N$. The corresponding vector spaces are denoted by $M_\RR$ and $N_\RR$, respectively.

We will describe Gorenstein Fano T-varieties of complexity $1$ by the following set of data. A lattice polytope $\Box$ in $M_\RR$, which contains the origin, together with a concave function 
\[\Phi \colon \Box \rightarrow \wdiv_\RR \PP^1, \quad
u \mapsto \sum_{y \in \PP^1} \Phi_y(u) \cdot y,
\] 
such that 
\begin{enumerate}
\item $\Phi$ is piecewise affine linear, i.e. given by affine linear functions on a finite polyhedral subdivision of $\Box$,
\item for $y\in \PP^1$, the graph of $\Phi_y$ has integral vertices,
\item for every $u$ in the interior of $\Box$, $\deg \Phi(u) > -2$.
\item The affine linear pieces of $\Phi_y$ have the form $u \;\mapsto\;  \frac{\langle v, u \rangle-\mu+1}{\mu}$, where $v\in N$ is a primitive lattice element. 
\label{item:distance-vertical};
\item every facet $F$ of $\Box$ with $(\deg \circ \Phi)|_F \not \equiv -2$ has lattice distance $1$ from the origin. \label{item:lattice-distance-horizontal}
\end{enumerate}

This gives rise to a polarised $T$-variety of complexity one in the following way. Consider 
$\bar \Phi \colon \Box  \rightarrow \wdiv_\RR \PP^1$ given by $\bar \Phi(u) = \Phi(u)+D$, where $D$ is some integral divisor of degree $2$. Then 
\begin{equation}
  H^0(X,L^{\otimes k})=\bigoplus_{u\in \Box\cap \frac{1}{k}M} H^0(\PP^1,\CO(\lfloor k\cdot \bar \Phi(u)\rfloor))\label{eq:sections}
\end{equation}
defines a polarised variety $(X,L)$. It is easy to see that a different divisor $D'$ of degree $2$ will give rise to the same polarised variety, since $D-D'$ is principal. Morover, associating  $H^0(\PP^1,\CO(\lfloor k\cdot \bar \Phi(u)\rfloor))$ the weight $ku \in M$ induces an $M$-grading on the section ring of $L$ and, hence, a $T$-action on $X$. 

Indeed, the function $\bar \Phi$ was called a Fano divisorial polytope in \cite{tfano,is15} and it was shown there that $(X,L)$ is a Gorenstein canonical variety polarised by its ample anticanonical line bundle.


\begin{example}[Cubic surface]
\label{exp:running} 
 Consider the map $\Phi:[-1,3] \to \wdiv_\RR(\PP^1)$ given by
  \[
  \Phi(u) = \Phi_0(u) \cdot \{0\} + \Phi_\infty(u) \cdot \{\infty\} + \Phi_1(u) \cdot \{1\}.
  \]
  with
  $\Phi_0(u) = \min \{-u,0\}$, $\Phi_\infty(u) = \frac{u-3}{4}$ and $\Phi_1(u) = \frac{u-1}{2}$.
  This fulfils the conditions (i)-(v) above.  Note, that we have $\deg \Phi(-1)=\deg \Phi(3)=-2$ and part~(\ref{item:lattice-distance-horizontal}) of the conditions is void in this case.

Now, we are interested in a description of the associated section ring. Choose $\bphi=\Phi + 2\cdot\{0\}$ then
\begin{align*}
  \bphi(-1) &= 2\{0\} -\{\infty\} -\{1\}\\
  \bphi(0) &= 2\{0\} -\frac{3}{4}\{\infty\} -\frac{1}{2}\{1\}\\
  \bphi(1) &= \{0\} -\frac{2}{4}\{\infty\} \\
  \bphi(2) &=  -\frac{1}{4}\{\infty\} + \frac{1}{2}\{1\}\\
  \bphi(3) &= -\{0\} + \{1\}\\
\end{align*}
After rounding down the corresponding divisors have degree $0$ with the exception of $\lfloor \bphi(2)\rfloor$, which has degree $-1$. Hence, in degree $k=1$ with respect to the usual $\ZZ$-grading of the section ring we find the following four generators, where $z^k$ keeps track of this degree and $\chi^u$ keeps track of the weight $u \in M \cong \ZZ$ as in (\ref{eq:sections}).
\[x_0 = (y-1)y^{-2} \cdot z\chi^{-1},\;
  x_1 = (y-1)y^{-2} \cdot z\chi^{0},\;
  x_2 = y^{-1} \cdot z\chi^{1},\;
  x_3 = y/(y-1) \cdot z\chi^{3}\]
On can check that these elements generate the section ring. The relations are generated by $x_0x_2x_3 + x_1^2x_3 + x_2^3$. Hence, we obtain a cubic surface.
\end{example}

As in the toric case it is possible to read off many properties and invariants of the variety directly from $\Phi:\Box \to \wdiv_\RR\PP^1$, see e.g. \cite{tfano}. Here, we are mainly interested in the Fano degree and in the Cox ring.

The Fano degree is the top self-intersection number of the anti-canonical divisor. It can be calculated from the combinatorial data by the following 

\begin{theorem}[{\cite[Proposition 3.31]{tidiv}}]
\label{thm:degree} For the Fano degree one obtains
  \[(-K_X)^{n} = n! \int_\Box \left(\deg \bar \Phi \right) = n! \int_\Box \left(2+ \deg \Phi \right).\]
\end{theorem}

\bigskip

For a normal complete variety $X$ the Cox ring is defined (as a graded vector space) by
\[
\Cox(X) = \bigoplus_{[D] \in \Cl(X)} H^0(X, \CO(D)).
\]
It can be equipped with a natural multiplication map making it into a $\Cl(X)$-graded $\CC$-algebra, see \cite{zbMATH05518639} for details.

We sketch how to obtain generators and relations of the Cox ring of $X$ from the corresponding divisorial polytope. For this we need to introduce some notation. 

Let $\Vert_y$ be the set of facets of the graph of $\Phi_y$ and \[\Vert = \coprod_{y \in \PP^1} \Vert_y.\] Recall from condition~(\ref{item:distance-vertical}) above, that for $F \in \Vert$ the facet $F$ is the graph of an affine linear function of the following form
\[\frac{\langle v_F, \cdot \rangle-\mu_F+1}{\mu_F},\]
defined on a subset of $\Box$. By $\Hor$ we denote be the set of facets $G$ of $\Box$ such that $\deg \Phi|_G \not \equiv -2$.

\begin{theorem}[{\cite[Thm.~40]{t-survey}}]
\label{thm:cox-ring}
  \[\Cox(X)=\frac{\CC\left[T_{F},S_G \mid F \in \Vert, G \in \Hor\right]}{\langle T^{\mu(0)} + cT^{\mu(\infty)} + T^{\mu(c)} \mid c \in \CC^* \rangle},\]
where \(T^{\mu(y)}:=\prod_{F \in \Vert_y} T_{F}^{\mu_F}.\)
\end{theorem}

\begin{example}[Cubic surface -- continued]
\label{exp:cox-degree}
  Consider the divisorial polytope from Example~\ref{exp:running}. An elementary calculation shows that $\int_{-1}^3 (2+\deg \Phi(u))\; du = \nicefrac{3}{2}$. Hence, we verify by Theorem~\ref{thm:degree} that $X$ has Fano degree $3$.

Now, consider the Cox ring. For $\Phi_0$ the graph has two facets $F_1$ and $F_2$ corresponding to the affine linear pieces $u \mapsto -u$ and $u \mapsto 0$. This gives $\mu_{F_1} = \mu_{F_2}=1$ and $T^{\mu(0)} = T_{F_1}T_{F_2}$. For $\Phi_\infty$ there is a unique facet $F_3$ with associated affine linear form $u \mapsto \frac{u-3}{4}$. Hence, $\mu_{F_3}=4$ and $T^{\mu(\infty)}=T_{F_3}^4$ and similarly $T^{\mu(1)}=T_{F_4}^2$.
  
We obtain the following equation for the Cox ring.
\[\Cox(X)=\CC[T_{F_1},T_{F_2},T_{F_3},T_{F_4}]/\langle T_{F_1}T_{F_2} + T_{F_3}^4 + T_{F_4}^2 \rangle.\]
\end{example}

\medskip

By \cite[Theorem 4.3.]{is15} equivariant non-trivial special test configurations $\X$ of Gorenstein Fano T-varieties of complexity $1$ are given by the choice of $m \in \NN$, $v \in N$ and $y \in \PP^1$, such that $\Phi_z(0)$ is non-integral for at most one $z \neq y$. We call such a choice of $y$ \emph{admissible}. It easy to describe the toric special fibre $\X_0$ of such a test configuration. It corresponds to the polytope $\Delta_y \subset M_\RR \times \RR$ given by
\begin{equation}
\Delta_y = \Big\{(u,a) \in M_\RR \times \RR \; \Big| \; u \in \Box,\; -1-\sum_{z \neq y} \Phi_z(u) \leq a \leq 1+\Phi_y(u)\Big\}\label{eq:special-fibre}
\end{equation}

and the induced $\CC^*$-action on the special fibre $\X_0$ is given by the one-parameter subgroup of 
 $T' = T \times \CC^*$ corresponding to $v'=(-mv,m) \in N \times \ZZ$. We denote the corresponding test configuration by $\X = \X_{y,v,m}$ with special fibre $\X_0$.



As observed in \cite{is15} for the Futaki character of $(\X_0,\xi')$ one obtains
\begin{equation}
F_{\X_0, \xi'}(v') = \frac{1}{\vol \Delta_y}\left(\int_{\Delta_y} \langle u', v' \rangle \cdot e^{\langle u', \xi'\rangle} du'\right),\label{eq:futaki-character}
\end{equation}
with $\xi',v' \in N_\RR \times \RR$. On the other hand, for $v,\xi \in N_\RR$ one obtains
\begin{equation}
F_{X, \xi}(v) = F_{\X_0, (\xi,0)}((v,0))
= \frac{1}{\int_\Box \deg \bar \Phi(u) \,du }\left(\int_{\Box} \langle u, v \rangle \cdot \deg \bar \Phi(u) \cdot e^{\langle u, \xi \rangle}\, du\right)
, \label{eq:futaki-general-fibre}
\end{equation}
To see this consider the projection $M \times \ZZ \to M$ giving a coarsened grading on the sections of $(\X_0,\L_0)$. With this grading for $u \in M$ one has

\begin{align*}
\dim H^0(X,L^{\otimes k})_u &= \dim H^0(Y,\CO(\lfloor k\cdot \bar \Phi(u)\rfloor)) \\
&=  2 + \lfloor k\cdot \Phi_y(u) \rfloor + \sum_{z\neq y}  \lfloor k \cdot \Phi_z(u)\rfloor\\
&=  2 + \left\lfloor k\cdot \Phi_y(u) \right\rfloor + \left\lfloor \sum_{z\neq y}   k \cdot \Phi_z(u)\right\rfloor\\
&= \dim H^0(\X_0,\L_0^{\otimes k})_u.
\end{align*}
Since the definition of $F_{X, \xi}$ in (\ref{eq:3}) does only depend on these dimension counts 
and the values $e^{\langle u, \xi \rangle} = e^{\langle (u,a), (\xi,0) \rangle}$  the first equality in (\ref{eq:futaki-general-fibre}) follows. For the second equality note, that for $\xi'=(\xi,0)$ and $v'=(v,0)$ the integrand in (\ref{eq:futaki-character}) does not depend on the second factor of $M_\RR \times \RR$. Hence, integrating along this factor first gives just the height of $\Delta_y$ at $u$, which is exactly $\deg \bar \Phi(u)$. Hence, one obtains the integral in (\ref{eq:futaki-general-fibre}).

Now, for a pair $(X,\xi)$ to be equivariantly K-stable the Futaki invariant $F_{X, \xi}(v)$ has to vanish for every choice $v$ by the condition for product test configurations. With exactly the same arguments as in \cite[Section~3.1]{zbMATH05366644} we may see that there always exist a unique choice $\xi \in N_\RR$ for which $F_{X, \xi}$ is trivial. We call this $\xi$ a \emph{soliton canditate}. To see whether with this candidate $(X,\xi)$ is indeed equivariantly K-stable it remains to check positivity of 
\[\DF_\xi(\X_{y,0,1})=F_{\X_0, (\xi,0)}((0,1)) = \frac{1}{\vol \Delta_y}\int_{\Delta_y} \langle u', (0,1) \rangle \cdot e^{\langle u', (\xi,0)\rangle} \, du'\]
for every admissible choice of $y$. Note, that $\Delta_y = \Delta_{y'}$ for $y,y' \notin \supp \Phi$. This leaves us with a finite number of integrals to check. However, in general we cannot hope to find an exact solution for $\xi$. We have to deal with sufficiently good approximations, instead.

 \begin{remark}
\label{rem:no-specials}
   As a consequence of the above classification of non-product
   $T$-equivariant special test configurations we see that if there
   are at least $3$ points $y \in \PP^1$ with $\Phi_y(0)$ being
   non-integral then such a test configuration does not exists due to the lack of an admissible choice for $y \in \PP^1$. Hence, we obtain equivariant K-stability for the soliton candidate for free. 
 \end{remark}

\begin{example}[Cubic surface -- continued]
\label{exp:calculations}
 Recall the piecewise linear function $\Phi:[-1,3] \to \wdiv_\RR(\PP^1)$ from Example~\ref{exp:running} given by
\[\Phi_0(u) = \min \{-u,0\},\; \Phi_\infty(u) = \frac{u-3}{4},\; \Phi_1(u) = \frac{u-1}{2}.\]
Now, with the choice of $y=\infty$ the construction from \cite{is15} gives a test configuration with special fibre corresponding to the polytope $\Delta_\infty = \conv((-1,0),(0,-\nicefrac{1}{2}),(3,1))$ with induced $\CC^*$-action given by the one-parameter subgroup $(0,1) \in N \times \ZZ = \ZZ \times \ZZ$.

\begin{figure}[h]
  \centering
 \begin{tikzpicture}[scale=0.5]
    \draw[dotted,step=1,gray] (-1,-1) grid (3,1); \draw (-1,0) --
    (3,1) -- (0,-0.5)--(-1,0);
  \end{tikzpicture}
\caption*{The polytope $\Delta_\infty$}
\end{figure}
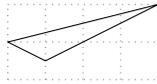

This degeneration can be realised in $\PP(1,1,1,1,2)$ by the equation
\[
\alpha \cdot x_4+x_1^2+x_0x_2,\; x_2^3 + x_3x_4,
\] 
where $\alpha \in \CC$  is the parameter of the degeneration. Note, that for $\alpha \neq 0$ we may eliminate $x_4$ and get the equation from Example~\ref{exp:running} (up to scaling of variables). This degeneration is induced by the action of $\CC^*$ with weights $(0,0,0,-1,1)$ on 
\[ V(x_4+x_1^2+x_0x_2,\; x_2^3 + x_3x_4) \subset \PP(1,1,1,1,2).\]

To check for the existence of a vector field $\xi \in N_\RR$, such that $(X,\xi)$ is K-stable, we first have to determine the unique candidate with sufficient precision. Hence, we have to approximate a solution of 

\begin{equation*}
0=f(\xi) := \vol \Delta_\infty \cdot  F_{X, \xi}(1) = \int_{\Delta_\infty} u_1 \cdot e^{\xi \cdot u_1} \; du_1du_2.
\end{equation*}

Solving the integral leads to
\begin{equation}
0=f(\xi)=-\nicefrac{3}{4}(\xi + 2)e^{-\xi} + \nicefrac{1}{4}((3\xi - 2)e^{3\xi} + 8).\label{eq:Fexample}
\end{equation}

When evaluating the exponential functions occuring in (\ref{eq:Fexample}) with an guaranteed precision of 11 binary digits at the values $-1.247$ and $-1.246$, it can be shown by elementary estimations (e.g. interval arithmetic) that $f(-1.247)<0$ and $f(-1.246)>0$.  The intermediate value theorem then implies $-1.247< \xi < -1.246$ for the solution $\xi$ of (\ref{eq:Fexample}).

Now it remains to check the sign of the integral
\[
\int_{\Delta_y} u_2 \cdot e^{\xi\cdot u_1} \; du_1du_2 \quad = \quad \frac{1}{16\xi^3}((4\xi - 3)e^{3\xi} + 8\xi + 3e^{-\xi}).\]
Using the same approximations as above for the exponential functions evaluated at the lower and upper bounds for $\xi$ gives an estimate 
\[-0.012 \leq\int_{\Delta_y} u_2 \cdot e^{\xi\cdot u_1} \leq -0.005.\]
In any case this shows that the Donaldson-Futaki invariant is negative and the pair $(X,\xi)$ gets destabilised.
\end{example}

\begin{remark}
Although it is in principle possible to do the calculations/estimates in Example~\ref{exp:calculations} by hand it becomes quite tedious. Hence,  we used interval arithmetic library MPFI \cite{mpfi} via the SageMath \cite{sage} computer algebra system to verify K-stability for our example, see Appendix~\ref{sec:sagemath-code-surface}. 
\end{remark}

\begin{remark}
  As for Example~\ref{exp:calculations} the standard integrals appearing in (\ref{eq:futaki-character}) and (\ref{eq:futaki-general-fibre}) can be solved analytically in general, either by elementary methods or by using Stoke's theorem to reduce to similar integrals along the boundary facets and by iterating this process eventually obtaining a formula which involves evaluations of exponential functions in the vertices of the polytope and rational functions, see \cite[Lemma 1]{zbMATH00525648}. 
\end{remark}

\section{Classification}
\label{sec:classification}
In this section we are considering all Gorenstein del Pezzo surfaces, which admit a non-trivial $\CC^*$-action. We give a list of the corresponding combinatorial data in Table~\ref{longtable}. For every del Pezzo surface we state the closed interval $\Box$, the functions $\Phi_y\colon \Box \to \RR$ with $y \in \supp \Phi$. One can check that these data fulfil the conditions (i)-(iv) from above and, hence, define Gorenstein del Pezzo surfaces with $\CC^*$-action. Note, that for all considered surfaces the support of $\Phi$ contains either $3$ or $4$ points. Hence, with the appropriate choice of coordinates on $\PP^1$ we may assume that the support consists of the points $0$, $\infty$, $1$ and possibly a fourth point $c$. If the support consists of $4$ points the combinatorial data gives rise to a 1-parameter family of de Pezzo surfaces parametrised by $c \in \PP^1 \setminus \{0,1,\infty\}$.

A classification of Gorenstein del Pezzo surfaces with $\CC^*$-action in terms of their Cox-rings is available from \cite[Section~5.3]{elaine}.  One the other hand, Theorem~\ref{thm:cox-ring} and Theorem~\ref{thm:degree} allow us to compare the Cox rings and Fano degrees of the del Pezzo surfaces in Table~\ref{longtable} with those in \cite{elaine}. Doing so we see that the list below is complete and hence provides the combinatorial description of all Gorenstein del Pezzo surfaces with $\CC^*$-action. Recently, the same classification was obtained in \cite{IMC} directly in terms of divisorial polytopes.
{\small
\setlength\LTleft{-1.4cm}
\begin{longtable}{=r+c+c+l+l+c+c+c}
\caption{Gorenstein del Pezzo surfaces with $\CC^*$-action}
\label{longtable}
\endfirsthead
\endhead
\toprule
   No.                      & K-stab   & $\xi$        & $\Box$   & $\Phi_{0},\Phi_{\infty},\Phi_1, (\Phi_{c})$                                       & $(K_X)^2$ & Sing.     & $\rho$ \\  \midrule
  \rowstyle{\color{gray}} 1 & \cmark  & $0$          & $[-1,1]$ & $\min \{-2u,-u\},\ \frac{u-1}{2},\ \frac{u-1}{2},\ \frac{u-1}{2} $                & $1$       & $2D_4$    & $1$    \\
   2                        & \cmark  & $-1.99761$   & $[-1,5]$ & $\min \{\frac{u-4}{5},0\} ,\ \frac{u-2}{3},\  \frac{-u-1}{2} $                    & $1$       & $E_8$     & $1$    \\
   3                        & \cmark  & $ -1.94024$  & $[-1,3]$ & $\min \{\frac{u-2}{3},0\}  ,\ \frac{u-3}{4},\  \frac{-u-1}{2} $                   & $1$       & $E_7A_1$  & $1$    \\
   4                        & \cmark  & $ -1.69131$  & $[-1,2]$ & $\min \{\frac{u-1}{2},0\}  ,\ \frac{u-2}{3},\  \frac{-2u-2}{3} $                  & $1$       & $E_6A_2$  & $1$    \\
  \midrule                                                        
  \rowstyle{\color{gray}} 5 & \cmark  & $0$          & $[-1,1]$ & $\min \{-2u,0\}  ,\ \frac{u-1}{2}   ,\  \frac{u-1}{2}$                            & $2$       & $2A_3A_1$ & $1$    \\
   6                        & \cmark  & $ -0.97052$  & $[-1,2]$ & $\min \{-u,0\}  ,\ \frac{u-2}{3}   ,\  \frac{u-2}{3}$                             & $2$       & $A_5A_2$  & $1$    \\
   7                        & \cmark  & $ -1.79675$  & $[-1,3]$ & $\min \{\frac{u-1}{2},0\}  ,\ \frac{u-3}{4},\  \frac{-u-1}{2}$                    & $2$       & $D_6A_1$  & $1$    \\
   8                        & \cmark  & $ -1.99186$  & $[-1,5]$ & $\min \{\frac{u-3}{4},0\}  ,\ \frac{u-2}{3},\  \frac{-u-1}{2}$                    & $2$       & $E_7$     & $1$    \\
   9                        & \cmark  & $ -1.94024$  & $[-1,3]$ & $\min \{\frac{u-2}{3},0\}  , \   \min \{\frac{u-2}{3},0\}  , \   \frac{-u-1}{2} $ & $2$       & $E_6$     & $2$    \\
                                                                  
\rowstyle{\color{gray}} 10  & \cmark  & $0$          & $[-1,1]$ & $\min \{-u,0\}  , \   \min \{-u,0\}  , \   \frac{u-1}{2} ,\   \frac{u-1}{2} $        & $2$ & $2A_3$    & $2$ \\    
 11                         & \cmark  & $ -1.69131$  & $[-1,2]$ & $\min \{\frac{u-1}{2},0\}   , \   \min \{\frac{u-1}{2},0\}   , \   \frac{-2u-2}{3} $ & $2$ & $D_5A_1$  & $2$ \\
 12                         & \cmark  & $ -1.34399$  & $[-1,1]$ & $\frac{u-1}{2},\ \frac{u-1}{2},\ \frac{-u-1}{2}$                                     & $2$ & $D_43A_1$ & $1$ \\
\midrule                                                          
 13                         & \xmark & $ -1.24607$  & $[-1,3]$ & $\min \{-u,0\}  ,\ \frac{u-3}{4}   ,\  \frac{u-1}{2} $                               & $3$ & $A_5A_1$  & $1$ \\
 14                         & \cmark  & $ -1.96766$  & $[-1,5]$ & $\min \{\frac{u-2}{3},0\} ,\ \frac{u-2}{3},\  \frac{-u-1}{2}$                        & $3$ & $E_6$     & $1$ \\
 15                         & \cmark  & $ -1.19618$  & $[-1,2]$ & $\min\{\frac{u-1}{2},0\}   , \   \min\{-u,0\}   , \   \frac{u-2}{3} $                & $3$ & $A_4A_1$  & $2$ \\
\rowstyle{\color{gray}} 16  & \xmark & $0$          & $[-1,1]$ & $\min\{-u,u\}  , \   \min\{-u,0\}   , \   \min \frac{u-1}{2} $                       & $3$ & $2A_2A_1$ & $2$ \\
 17                         & \cmark  & $ -1.83879$  & $[-1,3]$ & $\min \{\frac{u-2}{3},0\}   , \   \min\{\frac{u-1}{2},0\}   , \   \frac{-u-1}{2}$    & $3$ & $D_5$     & $2$ \\
\rowstyle{\color{gray}} 18  & \xmark & $0$          & $[-1,1]$ & $\min \{0,u\},\ \min\{-u,0\},\ \min \{-u,0\},\ \frac{u-1}{2}$                        & $3$ & $2A_2$    & $3$ \\
 19                         & \cmark  & $ -1.69131$  & $[-1,2]$ & $\min \{-u,\frac{-u-1}{2}\},\ \min\{0,\frac{u-1}{2}\},\ \min\{0,\frac{u-1}{2}\}$     & $3$ & $D_4$     & $3$ \\
 20                         & \cmark  & $ -0.94468$  & $[-1,1]$ & $\min\{0,u\},\ \frac{u-1}{2},\ \frac{-u-1}{2}$                                       & $3$ & $A_32A_1$ & $2$ \\
\midrule                                                          
 21                         & \cmark  & $ -1.85969$  & $[-1,5]$ & $\min \{\frac{u-1}{2},0\}  ,\ \frac{u-2}{3},\  \frac{-u-1}{2} $                      & $4$ & $D_5$     & $1$ \\
 22                         & \xmark & $ -0.97052$  & $[-1,2]$ & $\min\{0,u\}  , \   \min\{0,u\}   , \   \frac{-2u-2}{3}$                             & $4$ & $A_3A_1$  & $2$ \\
 23                         & \cmark  & $ -1.79675$  & $[-1,3]$ & $\min\{\frac{u-1}{2},0\}   , \   \min\{\frac{u-1}{2},0\}   , \   \frac{-u-1}{2}$     & $4$ & $D_4$     & $2$ \\
 24                         & \cmark  & $ -1.38176$  & $[-1,3]$ & $\min \{\frac{u-2}{3},0\}  , \   \min \{u,0\}   , \   \frac{-u-1}{2} $               & $4$ & $A_4$     & $2$ \\
 \rowstyle{\color{gray}} 25 & \xmark & $0$          & $[-1,1]$ & $\min \{0,2u\},\ \min \{-u,0\},\ \min\{-u,0\}$                                       & $4$ & $3A_1$    & $3$ \\
 26                         & \cmark  & $ -1.31047$  & $[-1,2]$ & $\min \{-u,0\},\ \min \{0,\frac{u-1}{2}\},\ \min \{0,\frac{u-1}{2}\}$                & $4$ & $A_3$     & $3$ \\
\rowstyle{\color{gray}} 27  & \cmark  & $0$          & $[-1,1]$ & $\min \{0,u\},\  \min \{0,u\},\ \min \{-u,0\},\ \min \{-u,0\}$                       & $4$ & $2A_1$    & $4$ \\
 28                         & \xmark & $  -0.74373$ & $[-1,1]$ & $\min\{0,u\},\ \min\{0,u\},\ \frac{-u-1}{2}$                                         & $4$ & $A_2A_1$  & $3$ \\
\midrule                                                          
 29                         & \cmark  & $ -1.42059$  & $[-1,5]$ & $\min \{-u,0\}  ,\ \frac{u-2}{3}  ,\  \frac{u-1}{2} $                                & $5$ & $A_4$     & $1$ \\
 30                         & \cmark  & $ -1.43886$  & $[-1,3]$ & $\min \{\frac{u-1}{2},0\}  , \   \min\{u,0\}   , \   \frac{-u-1}{2}$                 & $5$ & $A_3$     & $2$ \\
 31                         & \xmark & $ -1.10613$  & $[-1,2]$ & $\min\{0,u\},\ \min\{0,u\},\ \min\{0,\frac{-u-1}{2}\}$                               & $5$ & $A_2$     & $3$ \\
 32                         & \cmark  & $ -0.61790$  & $[-1,1]$ & $\min \{0,u\},\ \min \{0,u\},\ \min \{-u,0\}$                                        & $5$ & $A_1$     & $4$ \\
\midrule                                                          
 33                         & \xmark & $ -1.24607$  & $[-1,3]$ & $\min\{0,u\} , \   \min\{0,u\}   , \   \frac{-u-1}{2} $                              & $6$ & $A_2$     & $2$ \\
 34                         & \cmark  & $ -0.97052$  & $[-1,2]$ & $\min \{-u,0\},\ \min \{0,u\},\ \min \{0,u\}$                                        & $6$ & $A_1$     & $3$ \\\bottomrule  
\end{longtable}}

\begin{example}[Cubic surface -- continued]
  Consider once again the cubic surface from the Examples~\ref{exp:running} and \ref{exp:cox-degree}. From Example~\ref{exp:cox-degree} we know that the Fano degree equals $3$ and the Cox ring is isomorphic to $\CC[T_{1},T_{2},T_{3},T_{4}]/\langle T_{1}T_{2} + T_{3}^4 + T_{4}^2 \rangle$.
Comparing this with the data from Theorem~5.25 in \cite{elaine} we find that the corresponding surface has Picard rank $\rho=1$ and singularity type $A_5A_1$.
\end{example}

\begin{proof}[Proof of Theorem~\ref{thm:surface-solitons}]
\label{sec:proof-of-surface-thm}
  We are running through the classification given in Table~\ref{longtable}. First note, that the cases 1, 5, 10 and 27 are known to admit K\"ahler-Einstein metrics, hence are K-stable, see \cite[Section~6.1.2]{zbMATH06546438}. For the cases 16, 18 and 25 one also calculates  $F_{X,0} = 0$. Hence, we have $\xi=0$ for the soliton candidate. On the other hand, these surfaces are known to be not K\"ahler-Einstein, see  loc. cit.

  For the remaining cases we follow the approach outlined in Example~\ref{exp:calculations}. Hence, we apply the following steps:
  \begin{enumerate}
  \item Find a closed form for $F_{X,\xi}(1)$ in terms of exponential functions in $\xi$. This can be done  by solving the integral (here over an interval) appearing in (\ref{eq:futaki-general-fibre}) analytically using standard methods. \label{item:solve-f}
  \item Find sufficiently good bounds $\xi_-$ and $\xi_+$ with $\xi_- < \xi < \xi_+$ for a solution $\xi$ of $F_{X,\xi}(1)=0$. To show that the interval $(\xi_-, \xi_+)$ contains a solution we
calculate $F_{X,\xi_-}(1)$ and $F_{X,\xi_+}(1)$ with sufficient precision (i.e. we need to guarantee error bounds for the approximation of the exponential function) and then use intermediate value theorem.
\label{item:approx-vf}
\item For every admissible choice of $y \in \PP^1$ find a closed form for $\DF_\xi(\X_{y,0,1})$ in terms of exponential functions. For this we have to analytically solve the integral in (\ref{eq:futaki-character}) for $\xi'=(\xi,0)$ and $v'=(0,1)$. This comes down to solving
\begin{align*}
\int_{\Delta_y} u_2 e^{u_1\xi} du_1du_2 &= \int_{\Box} e^{u_1\xi} \int\displaylimits_{-1-\sum_{z \neq y}\Phi_z(u_1)}^{1+\Phi_y(u_1)} \hspace{-2em}u_2\; du_2 \;\;du_1\\ 
& = \int_{\Box} e^{u_1\xi} (1+\Phi_y(u_1))^2 du_1 - \int_{\Box} e^{u_1\xi} (1+\sum_{z\neq y} \Phi_z(u_1))^2 \;du_1.
\end{align*}
Here, the right hand side just involves standard integrals in one variable and can be solved by elementary methods.\label{item:solve-df}
\item Ultimatively we have to plug in the value of $\xi$ into the found closed form and check positivity. However, we have only estimates for $\xi$ and also for the evaluations of the exponential functions appearing in the closed form obtained in (\ref{item:solve-df}). Hence, we need to use elementary estimations to ensure positivity for all values within the known error bounds.
\end{enumerate}
The complete calculations are done using SageMath and can be found in the ancillary files\cite{sage-code} and as an online worksheet\footnote{CoCalc:\url{https://cocalc.com/projects/ae8e1663-e2ad-40b8-aec2-30faf4e6a54f/files/surfaces.sagews}}. For an example which can be adopted to the other cases see also Appendix~\ref{sec:sagemath-code-surface}.

  Note, that as in Remark~\ref{rem:no-specials} for the cases no.~1, 2, 3, 4, 7, 8 , 9, 11, 12, 14, 17, 18, 19, 21 and 23 we obtain the existence of a K-stable pair $(X,\xi)$ without any calculation, since in these cases there is no admissible choice of $y$. However, to obtain an approximation for $\xi$ we still have to do the calculations in (\ref{item:solve-f}) and (\ref{item:approx-vf}). In particular, in all cases with non-trivial candidate vector fields $\xi$ we obtain a K-stable pair with the exception of nos. 13, 22, 28, 31 and 33. 
\end{proof}

The cases no. 13, 22, 28, 31, 33 do not admit a K-stable pair and, hence, no K\"ahler-Ricci soliton. We provide a description of the destabilising test configurations in Table~\ref{destab}. Indeed, from the calculations one obtains as destabilising test configurations $\X_{\infty,0,1}$ for no. 13 and $\X_{1,0,1}$ for all other cases. The description of the special fibre from (\ref{eq:special-fibre}) immediately provides the last column of the table. To obtain the equations and the ambient space we refer to the explicit construction of the test configuration given in \cite[Section~4.1]{is15}. The third column states the weights for the $\CC^*$-action on the ambient space, which induces the $\CC^*$-action on the total space of the test configuration.

\setlength\LTleft{-1.4cm}
\begin{longtable}{rlccl}
\caption{Destabilising test configurations}
\label{destab}
\endfirsthead
\endhead
\toprule
No.&destabilising degeneration&amb. space&weights&special fiber\\
\midrule
13&$\alpha \cdot x_4+x_1^2+x_0x_2,\; x_2^3 + x_3x_4$&$\PP(11112)$&$(0,0,0,-1,1)$&
  \raisebox{-0.3cm}{\begin{tikzpicture}[scale=0.4]
    \draw[dotted,step=1,gray] (-1,-1) grid (3,1);
    \draw (-1,0) -- (3,1) -- (0,-0.5)--(-1,0);
  \end{tikzpicture}}
\\
\midrule
22&$\alpha \cdot x_2x_3 + x_0x_1-x_2^2,x_3x_4-x_1^2$&$\PP^4$&$(0,0,0,-1,1)$&
  \raisebox{-0.3cm}{\begin{tikzpicture}[scale=0.4]
    \draw[dotted,step=1,gray] (-1,-1) grid (3,1);
    \draw (-1,1) -- (0,-1) -- (2,-1)--(-1,1);
  \end{tikzpicture}}
\\
\midrule
28&$\alpha \cdot x_2x_3 + x_0x_1-x_2^2,x_3x_4-x_1x_2$&$\PP^4$&$(1,-1,0,-1,0)$&
 \raisebox{-0.3cm}{\begin{tikzpicture}[scale=0.4]
    \draw[dotted,step=1,gray] (-1,-1) grid (3,1);
    \draw (-1,1) -- (0,-1) -- (1,-1)--(1,0) -- (-1,1);
  \end{tikzpicture}}
\\
\midrule
31&
$\begin{array}{l}
  \alpha \cdot x_3x_5 + x_3^2-x_0x_4,\\
  \alpha \cdot x_3x_4 + x_0x_2-x_1x_3,\\ 
  x_2x_3-x_1x_4,x_3x_4-x_1x_5, x_4^2-x_2x_5
\end{array}$&$\PP^5$&$(0,1,1,0,0,-1)$&
 \raisebox{-0.3cm}{\begin{tikzpicture}[scale=0.4]
    \draw[dotted,step=1,gray] (-1,-1) grid (3,1);
    \draw (-1,1) -- (0,-1) -- (2,-1)--(1,0) -- (-1,1);
  \end{tikzpicture}}
\\
\midrule
33& 
$\begin{array}{l}
\alpha\cdot x_3x_4+x_3^2-x_0x_6,\\
\alpha\cdot x_2x_3+x_1x_3-x_0x_5,\\
\alpha\cdot x_1x_2+ x_1^2-x_0x_4\\
x_5^2-x_4x_6, x_4x_5-x_2x_6, x_3x_5-x_1x_6,\\
x_4^2-x_2x_5, x_3x_4-x_1x_5, x_2x_3-x_1x_4
\end{array}
$
&$\PP^6$&$(1,0,-1,0,-1,-1,-1)$ & 
\raisebox{-0.3cm}{\begin{tikzpicture}[scale=0.4]
    \draw[dotted,step=1,gray] (-1,-1) grid (3,1);
    \draw (-1,1) -- (0,-1) -- (3,-1)-- (-1,1);
  \end{tikzpicture}}
\\
\bottomrule  
\end{longtable}

\begin{remark}
 Note, that some of the K-unstable examples seem to be closely related to each other. Indeed, no.~18, 28, 31 are (weighted) blowups of of no.~33 and no.~13 is a quotient of 33 by $\ZZ/2\ZZ$. Surface no.~16 lies at the boundary of the family of del Pezzo surfaces of type 18.
\end{remark}

\section{New K\"ahler-Ricci solitons on Fano threefolds}
\label{sec:threefolds}
In this section we consider Fano threefolds admitting an effective \(2\)-torus action within the classification of \cite{mori-mukai}. In \cite{tfano} a not necessarily complete list of such threefolds together with their combinatorial description was given. We use the methods described above to extend the results of \cite{is15}, providing new examples of threefolds  admitting a non-trivial K\"ahler-Ricci soliton.

 For this we use the same approach as in the proof of Theorem~\ref{thm:surface-solitons}. We determine a small region in $N_\RR$ which has to contain the soliton candidate $\xi$. Then we use these bounds for $\xi$ to show positivity of the Donaldson-Futaki invariants $\DF_\xi(\X_{y,0,1})$ for every admissible choice of $y$.  However, now $\xi$ is two-dimensional and we cannot use the intermediate value theorem directly to bound the solution for $\xi$. In some cases we can make use of additional symmetries to reduce to a one-dimensional problem. Here, the key observation is that given an automorphism \(\sigma \in \GL(M)\) permuting the vertices of \(\Box\) such that \(\deg (\Phi \circ \sigma) = \deg \Phi\), by (\ref{eq:futaki-general-fibre}) we have $F_{X, \sigma^{\!*}\!(\xi)}  =  F_{X, \xi}\circ \sigma^*$.
Since \(\xi \in N_\mathbb{R}\) is the unique solution to \(F_{X,\xi} = 0\), this gives \(\xi \in N_{\mathbb{R}}^{\sigma^*}\). Now we show how to utilise this observation.
\begin{example}[2.30 -- Blow up of quadric threefold in a point]
\label{exp:threefold}
Consider the threefold 2.30. The combinatorial data for this threefold was given in\cite{tfano}, although the piecewise affine \(\Psi\) discussed there is \(\bar{\Phi}\) as denoted in Section 2, with \(D = 2 \cdot \{ \infty \}\). The function \(\Phi\) is given in Figure~\ref{fig:data230}.
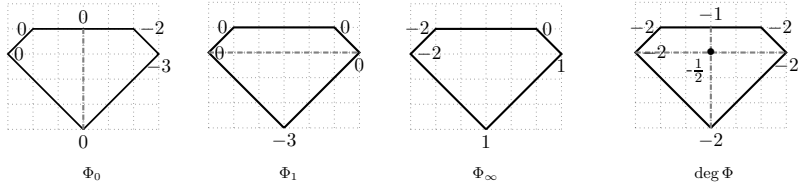
\begin{figure}[h]
\caption{The combinatorial data for threefold 2.30}
\label{fig:data230}
\resizebox{0.90\linewidth}{!}{
\begin{subfigure}[b]{0.30\textwidth}
\centering
  \begin{tikzpicture}[scale=0.5]
   	 \draw[dotted,step=1,gray,] (-3,-3) grid (3,2); \draw[line width = 1pt] (0,-3) --
     (-3,0) -- (-2,1)--(2,1)--
 	 (3,0)--(0,-3); \draw[densely dashdotted, gray, line width = 1.2pt] (0,-3) -- (0,1);
 	 \node at (0,-3) [below] {\large{$0$}};
 	 \node at (-3,0) [right] {\large{$0$}};
 	 \node at (-2,1) [left] {\large{$0$}};
 	 \node at (2,1) [right] {\large{$-2$}};
 	 \node at (3,0) [below] {\large{$-3$}};
 	 \node at (0,1) [above] {\large{$0$}};
	\end{tikzpicture}
	\caption*{$\Phi_0$}
\end{subfigure}
\begin{subfigure}[b]{0.30\textwidth}
	\centering
 	 \begin{tikzpicture}[scale=0.5]
 	   \draw[dotted,step=1,gray] (-3,-3) grid (3,2); \draw[line width = 1.2pt] (0,-3) --
 	   (-3,0) -- (-2,1)--(2,1)--
 	   (3,0)--(0,-3); \draw[densely dashdotted, gray,line width = 1.2pt] (-3,0) -- (3,0);
 	 \node at (0,-3) [below] {\large{$-3$}};
 	 \node at (-3,0) [right] {\large{$0$}};
 	 \node at (-2,1) [left] {\large{$0$}};
 	 \node at (2,1) [right] {\large{$0$}};
 	 \node at (3,0) [below] {\large{$0$}};
	\end{tikzpicture}
	\caption*{$\Phi_1$}
\end{subfigure}
\begin{subfigure}[b]{0.30\textwidth}
	\centering
 	 \begin{tikzpicture}[scale=0.5]
 	   \draw[dotted,step=1,gray] (-3,-3) grid (3,2); \draw[line width = 1.2pt] (0,-3) --
 	   (-3,0) -- (-2,1)--(2,1)--
 	   (3,0)--(0,-3);
 	 \node at (0,-3) [below] {\large{$1$}};
 	 \node at (-3,0) [right] {\large{$-2$}};
 	 \node at (-2,1) [left] {\large{$-2$}};
 	 \node at (2,1) [right] {\large{$0$}};
 	 \node at (3,0) [below] {\large{$1$}};
	\end{tikzpicture}
	\caption*{$\Phi_\infty$}
\end{subfigure}
\begin{subfigure}[b]{0.40\textwidth}
	\centering
 	 \begin{tikzpicture}[scale=0.5]
 	   \draw[dotted,step=1,gray] (-3,-3) grid (3,2); \draw[line width = 1.2pt] (0,-3) --
 	   (-3,0) -- (-2,1)--(2,1)--
 	   (3,0)--(0,-3); \draw[densely dashdotted, gray,line width = 1.2pt] (-3,0) -- (3,0); \draw[densely dashdotted, gray,line width = 1.2pt] (0,-3) -- (0,1) ;
 	 \node at (0,-3) [below] {\large{$-2$}};
 	 \node at (-3,0) [right] {\large{$-2$}};
 	 \node at (-2,1) [left] {\large{$-2$}};
 	 \node at (2,1) [right] {\large{$-2$}};
 	 \node at (3,0) [below] {\large{$-2$}};
 	 \node at (0,0) [below left] {\large{-$\frac{1}{2}$}};
 	 \node at (0,1) [above] {\large{$-1$}};
 	 \draw (0,0) node {\textbullet};
	\end{tikzpicture}
	\caption*{$\deg \Phi $}
\end{subfigure}
}
\end{figure}
We now find the unique candidate vector field \(\xi \in N_\mathbb{R}\) for a \(K\)-stable pair \((X,\xi)\). We see that $\deg \Phi$ is symmetric with respect to reflection $\sigma$ along the vertical axis. Hence, we have  \(\xi = \xi_2e_2\) for some \(\xi_2 \in \mathbb{R}\) and must find a solution $\xi_2$ to $F_{X,\xi_2e_2}=0$, which is equivalent to $F_{X,\xi_2e_2}(e_2)=0$. Indeed, we have 
\[F_{X,\xi_2e_2}(e_1) = F_{X,\sigma^*\xi_2e_2}(\sigma^*e_1)= F_{X,\xi_2e_2}(-e_1)= -F_{X,\xi_2e_2}(e_1).\]
Hence, $F_{X,\xi_2e_2}(e_1)=0$ and the claim follows by linearity.

By (\ref{eq:futaki-character}) the vanishing of $F_{X,\xi_2e_2}(e_2)$ is equivalent to that of
\[0=g(\xi_2) := \int_{\Box} u_2 \cdot \deg \bar \Phi(u) \cdot e^{u_2 \xi_2}\, du = \int_{\Delta_0} u_2 \cdot e^{u_2 \xi_2}\, du.\] Where the integral on the right hand side can be solved analytically. We obtain
\[
\frac{1}{\xi_{2}^{4}}\cdot\left({\left(2 \, \xi_{2}^{3} - 3 \, \xi_{2} - 3\right)} e^{\left(4 \, \xi_{2}\right)} + 12 \, \xi_{2} e^{\left(3 \, \xi_{2}\right)} + 3 \, \xi_{2} + 3\right) e^{\left(-3 \, \xi_{2}\right)}.
\]
Evaluating the exponential functions with a precision of 16 binary digits  and using elementary estimations it can be shown that \(g(0.514) <0\) and \(g(0.515)>0\). By the intermediate value theorem then \(0.514 < \xi_2 < 0.515\). It remains to check the positivity of the Donaldson-Futaki invariant for each degeneration. The degenerations of this threefold correspond to the polytopes:{
\begin{align*}
\Delta_0 &= \conv((-3,0,1),(-2,1,1),(2,1,-1),(3,0,-2),(0,-3,1),(0,1,1)); \\
\Delta_1 &= \conv((-3,0,1),(-2,1,1),(0,1,0),(2,1,1),(3,0,1),(0,-3,-2)); \\ 
\Delta_\infty &= \conv((-3,0,-1),(-2,1,-1),(2,1,1),(3,0,2),(0,-3,2),(0,0,-1),(0,1,-1)); \\
\Delta_y &= \conv((0, 0, -1/2), (3, 0, 1), (2, 1, 1), (0, 1, 0), (-2, 1, 1), (-3, 0, 1), (0, -3, 1)) \\ 
& \ \ \ ( \text{for } y 
\not \in \{0,1,\infty\} ).
\end{align*}}
In each case we have induced \(\CC^*\)-action given by \((0,0,1) \in N \times \ZZ \). Denote
\[
h_y(\xi_2) := (\vol \Delta_y) \cdot \DF_{(0,\xi_2)}(\X_{y,0,1}).
\]
Clearly positivity of \(h_y\) implies the positivity of \(\DF_\xi(\X_{y,0,1})\). 
Once more solving the integrals appearing in (\ref{eq:futaki-character}) analytically with $y \not \in \{0,1,\infty\}$  we obtain:
\begin{align*}
h_0(\xi_2) &= \frac{1}{3\xi_2^{4}}\cdot
{\left({\left(2   \xi_2^{3} - 3   \xi_2 - 3\right)} e^{4   \xi_2} + 3   {\left(3   \xi_2^{2} + 2\right)} e^{3   \xi_2} - 3   \xi_2 - 3\right)} e^{-3   \xi_2}
\\
h_1(\xi_2) &= \frac{1}{6   \xi_2^{4}}\cdot
{\left({\left(8   \xi_2^{3} + 6   \xi_2^{2} - 3\right)} e^{4   \xi_2} - 12   {\left(3   \xi_2^{2} - 3   \xi_2 + 1\right)} e^{3   \xi_2} + 12   \xi_2 + 15\right)} e^{-3   \xi_2} \\
h_\infty(\xi_2) &= -\frac{1}{6   \xi_2^{4}}\cdot
{\left(2   {\left(2   \xi_2^{3} - 3   \xi_2 - 3\right)} e^{4   \xi_2} - 3   {\left(3   \xi_2^{2} - 12   \xi_2 + 2\right)} e^{3   \xi_2} + 12   \xi_2 + 12\right)} e^{-3   \xi_2} \\
h_y(\xi_2) &= \frac{1}{6   \xi_2^{4}}\cdot{\left({\left(8   \xi_2^{3} + 6   \xi_2^{2} - 3\right)} e^{4   \xi_2 } - 3   {\left(3   \xi_2^{2} - 2\right)} e^{3   \xi_2} - 6   y - 3\right)} e^{-3   \xi_2} 
\end{align*}

Using the same precision as above for the  evaluations of the exponential functions at the lower and upper bounds for \(\xi_2\) gives estimates:
\begin{align*}
1.087 &< h_0(\xi_2) < 1.458 \\
2.178 &< h_1(\xi_2) < 2.470 \\
0.446 &< h_\infty(\xi_2) < 0.827 \\
4.151 &< h_y(\xi_2) < 4.309 \ \ \ \ \ \ \   \left( \text{for } y \not \in \{0,1,\infty\} \right)
\end{align*}
We can therefore conclude that the threefold 2.30 is \(K\)-stable, and must admit a non-trivial K\"ahler-Ricci soliton.
\end{example}

\begin{proof}[Proof of Theorem~\ref{thm:threefold-solitons}]
Looking at the combinatorial description for the threefolds given in \cite{tfano} one sees that for the cases no. 2.31, 3.18, 3.22, 3.24 and 4.8 $\deg \Phi$ admits a (non-trivial) involution $\sigma \in \GL(M)$. Then $\xi$ has to be contained in the line $N_\RR^{\sigma^*}$ and 
after choosing a basis  $e_1, e_2$ of $N_\RR$ with $\sigma^*(e_1)=-e_1$ and $\sigma^*(e_2)=e_2$ we may proceed as in Example~\ref{exp:threefold} with the following steps.
\begin{enumerate}
\item Find a closed form for $F_{X,\xi}(e_2)$. \label{item:closed-form}
\item Find sufficiently good bounds for $\xi_2$ with $F_{X,\xi_2e_2}(e_2)=0$ via intermediate value theorem.\label{item:bounds}
\item For every admissible choice of $y \in \PP^1$ find a closed form for $\DF_{\xi_2 e_2}(\X_{y,0,1})$.
  \label{item:formulae-for-DF}
\item Use elementary estimations to ensure positivity of $\DF_{\xi_2 e_2}(\X_{y,0,1})$ for all values of $\xi_2$ within the error bounds.
\end{enumerate}

For the case of threefold no. 3.23 there is no involution fixing $\deg \Phi$. In this case we take a more general approach to bound the value of the candidate \(\xi\). Here we make use of some elementary calculus. Note, that \(\xi\) is the unique solution to the equation \(\nabla G = 0 \), where 
\[
G(v) := \int_{\Box} \deg \bar \Phi(u) \cdot e^{\langle u, v \rangle}\, du = 
\int_{\Delta_0} e^{\langle u', (v,0) \rangle} \, du'.
\]
Now we identify a small closed rectangular region \(D \subset \RR^2 \) such that \(\nabla_n G > 0 \) holds along \(\partial D\), for \(n\) being a outer normal of the rectangle \(D\). This guarantees a solution to \(\nabla G = 0 \) in the interior of \(D\). Indeed, due to compactness, $D$ has to contain a minimum of $G$, which by our condition cannot lie on the boundary. Hence, the minimum is located in the interior and has to coincide with $\xi$, since $\nabla G$ necessarily vanishes. After bounding the value of $\xi$ we proceed with step~(\ref{item:formulae-for-DF}). 

However, for showing positivity of $\nabla_n G$ along $\partial D$ we have to use computer assistance. The approach is simple but computational intensive. First we again determine a closed form
for $\nabla G_n(\xi)$ which coincides with $F_{X,\xi}(n)$ up to a positive constant. Then we subdivide the faces of the boundary in sufficiently small segments, where one of coordinates is fixed and the other varies in a small interval. Using interval arithmetic when evaluating the closed form for $\nabla_n G(\xi)$ provides the positivity result. See also Example~\ref{exp:asymetric} for details of the computation and Appendix~\ref{sec:sagemath-code-asymmetric} for the implementation in SageMath.

The complete calculations are done using SageMath and can be found in the ancillary files\cite{sage-code} and as an online worksheet\footnote{CoCalc:\url{https://cocalc.com/projects/ae8e1663-e2ad-40b8-aec2-30faf4e6a54f/files/threefolds.sagews}}.
\end{proof}

\begin{example}[3.23 -- Blowup of the quadric in a point and a line passing through]
\label{exp:asymetric}
We follow the calculations outlined in the above proof of Theorem~\ref{thm:threefold-solitons}. As before we first have to find a closed form for $F_{X,\xi}(n)$ or $\nabla_n G(\xi)$, respectively. Then numerically we can find an approximation to \(\xi\) as the point:
\[
(x_0,x_1) = (0.26617786,  0.67164063).
\]
Setting \(\epsilon = 10^{-5}\), consider the square containing our approximation, given by:
\[
D = [x_0 - \epsilon,x_0+\epsilon] \times [x_1 - \epsilon, x_1 + \epsilon]
\]
Subdividing each edge of the boundary \(\partial D\) into line segments of length \(\epsilon/1500\), we use interval arithmetic to verify that the gradient of \(h\) is positive in the outer normal direction for each of these segments, in fact \(\nabla_n G > 5.536 \cdot 10^{-6}\) along \(\partial D\). Once again it remains to check the positivity of the Donaldson-Futaki invariant for each degeneration. The degenerations of this threefold correspond to the polytopes:
{
\
 \begin{align*}
\Delta_0 &= \conv((-3,0,1),(-2,1,1),(0,1,0),(0,1,1),(1,1,0),(2,0,-1),(2,-1,-1),\\
         &\qquad  (0,-3,1),(1,0,-1)); \\
\Delta_1 &= \conv((-3,0,1),(-2,1,1),(1,1,1),(2,0,1),(2,-1,0),(0,-3,-2),(0,1,0)); \\ 
\Delta_\infty &= \conv((-3,0,-1),(-2,1,-1),(0,1,0),(1,1,0),(2,0,1),(2,-1,2),(0,-3,2),\\
&\qquad (0,1,-1),(0,0,-1)) \\
\Delta_y &= \conv((0, 0, -1/2), (0, 1, 0), (1, 0, 0), (1, 1, 1), (2, 0, 1), (2, -1, 1), \\ & \qquad (-2, 1, 1), (-3, 0, 1), (0, -3, 1)) \ \ \  ( \text{for } y 
\not \in \{0,1,\infty\} ).
\end{align*}}
Interval arithmetic gives the following lower bounds on the Donaldson-Futaki invariants:
\begin{align*}
h_0(\xi_2) &> 1.2766 \\
h_1(\xi_2) &> 1.8401 \\
h_\infty(\xi_2) &> 0.1004 \\
h_y(\xi_2) &> 3.4443 \ \ \  ( \text{for } y 
\not \in \{0,1,\infty\} )
\end{align*}
We can therefore conclude that the threefold 3.23 is \(K\)-stable, and must admit a non-trivial K\"ahler-Ricci soliton. See also Appendix~\ref{sec:sagemath-code-asymmetric} for the SageMath code of the calculations.
\end{example}

\begin{remark}
  Note, that by Theorem~\ref{thm:threefold-solitons} and \cite[Thms. 6.1, 6.2]{is15} all known smooth Fano threefolds with complexity-one torus action admit a K\"ahler-Ricci soliton.
\end{remark}

\renewcommand{\thefootnote}{\fnsymbol{footnote}} 
\setcounter{footnote}{0}

In the Table~\ref{table:name} below we give the estimates found for the vector field \(\xi\) for each threefold in the list of \cite{tfano}. The threefolds 3.8\savefootnote{one-of-family}{This refers only to a particular element of the family admitting a 2-torus action}, 3.21, 4.5\repeatfootnote{one-of-family} were shown to admit a non-trivial K\"ahler-Ricci soliton in \cite{is15}. Applying 
steps (\ref{item:closed-form})-(\ref{item:bounds}) from the proof of Theorem~\ref{thm:threefold-solitons} provides also an approximation for the vector field \(\xi\) for these threefolds. These are included in the table, together with those threefolds shown in \cite{is15} to be K\"ahler-Einstein, to show the complete picture for the Fano threefolds described in \cite{tfano}. We can show that our approximations are correct to the nearest \(10^{-5}\).
\begin{table}[h]
\captionsetup{width=.95\linewidth}
\caption{Fano threefolds and their soliton vector fields in the canonical coordinates coming with the representation of the combinatorial data in \cite{tfano}.}
\begin{tabular}{=l+l+l}
\toprule
Threefold & $\xi$ & \\ \hline
\rowstyle{\color{gray}}
Q & $(0,0)$ \\
\rowstyle{\color{gray}}
2.24\repeatfootnote{one-of-family} & $(0,0)$ \\
\rowstyle{\color{gray}}
2.29 & $(0,0)$ \\
2.30 & $(0,0.51489)$ \\
2.31 & $(0.28550,0.28550)$\\
\rowstyle{\color{gray}}
2.32 & $(0,0)$ \\
3.8\repeatfootnote{one-of-family} & $(0,-0.76905)$ \\
\rowstyle{\color{gray}}
3.10\repeatfootnote{one-of-family} & $(0,0)$ \\
3.18 & $(0,0.37970)$ \\
\rowstyle{\color{gray}}
3.19 & $(0,0)$ \\
\rowstyle{\color{gray}}
3.20 & $(0,0)$ \\
3.21 & $(-0.69622,-0.69622)$ \\
3.22 & $(0,0.91479)$ \\
3.23 & $(0.26618,  0.67164)$ \\
3.24 & $(0,0.43475)$ \\
\rowstyle{\color{gray}}
4.4 &  $(0,0)$ \\
4.5\repeatfootnote{one-of-family} &  $(-0.31043,-0.31043)$ \\
\rowstyle{\color{gray}}
4.7 &  $(0,0)$ \\
4.8 &  $(0,0.62431)$ \\
\bottomrule
\end{tabular}
\label{table:name}
\end{table}



\appendix
\section{SageMath code for  examples}
\small
{\fontseries{lc}\selectfont
\setlength{\parindent}{0pt}

\subsection{No. 13: Degree $3$ / Singularity type $A_5A_1$.}
\label{sec:sagemath-code-surface}
 The combinatorial data is is given by $\Box=[-1,3]$ and $\Phi_0(u) = \min \{-u,0\},\; \Phi_\infty(u)= \frac{u-3}{4},\; \Phi_1(u) = \frac{u-1}{2}$.

\begin{pyin}
\ 
# We are working with interval arithmetic with  precision of 11 bit. This fixes
# the precision e.g. for evaluations of exponential functions
RIF=RealIntervalField(11)
# The base polytope/interval 
B=[-1,3]
# 3 PL functions on the interval [-1,3]
Phi1(u)=min_symbolic(-u,0)
Phi2(u)=1/4*u-3/4
Phi3(u)=1/2*u-1/2
# the degree of  Phi
degPhi=Phi1+Phi2+Phi3+2

\end{pyin}

{\bf Step (i) -- obtain a closed form for $F_{X,\xi}(1)$}\\
For this we have to analytically solve the integral $\int_{-1}^3 u \deg \bar \Phi (u) e^{\xi u} du$.%

\begin{pyin}[]
# The integral can be solved symbolically:
F(xi)=integral(degPhi(u)*u*exp(xi*u),u,*B)
show(F(xi))
\end{pyin}

\begin{pyprint}
$-\frac{3 \, {\left(\xi + 2\right)} e^{\left(-\xi\right)}}{4 \, \xi^{3}} + \frac{{\left(3 \, \xi - 2\right)} e^{\left(3 \, \xi\right)} + 8}{4 \, \xi^{3}}$
\end{pyprint}

{\bf Step (ii) -- find an estimate for the soliton candidate vector field $\xi$}
\begin{pyin}[]
# Choose upper and lower bound and hope that the exact solution lies in between
lower=-1.247; upper=-1.246
# define a real value xi0 between xi_1 and xi_2 representing the exact solution
xi0=RIF(lower,upper)
# Check whether Intermediate value theorem guarantees a zero between $\color{mycomment} \xi_-$ and 
# $\color{mycomment} \xi_+$, i.e evaluate F at $\color{mycomment} \xi_-$ and $\color{mycomment} \xi_+$ using interval arithmetic
if RIF(F(lower)) < 0 and RIF(F(upper)) > 0:
    print "Interval containing solution:", xi0.str(style='brackets')
\end{pyin}

\begin{pyprint}
  Interval guaranteed to contain solution: [-1.2471 .. -1.2451]
\end{pyprint}

{\bf Step (iii) \& (iv) -- obtain closed forms for $\operatorname{DF}_\xi(X_{y,0,1})$ and plug in $\xi$}\\
First have to symbolically solve the integrals  
$\int_\Box   e^{u\xi} \cdot \left((1+\Phi_y(u))^2 - (1+\sum_{z\neq y} \Phi_z(u))^2\right) du$ (which equal $\DF_\xi(\mathcal{X}_{y,0,1})$ up to scaling by a positive constant)
for every admissible choice of $y \in \mathbb P^1$ and then plug in the estimate for $\xi$ into the resulting expression.

\begin{pyin}[]
print("Stability test for test configurations:")
# There are 2 admissable choices $\color{mycomment} y=\infty$ and $\color{mycomment} y=1$
DF1(xi)=integral(exp(xi*u)*((Phi2+1)(u)^2-(Phi1+Phi3+1)(u)^2)/2,u,-1,B[1])
IDF1=RIF(DF1(xi0)) # evaluate DF1 at xi0 using interval arithmetic
show(DF1(xi) >= IDF1.lower())
DF2(xi)=integral(exp(xi*u)*((Phi3+1)(u)^2-(Phi1+Phi2+1)(u)^2)/2,u,-1,B[1])
IDF2=RIF(DF2(xi0)) # evaluate DF2 at xi0 using interval arithmetic
show(DF2(xi) >= IDF2.lower())

if IDF1 > 0  and IDF2 > 0: print("Surface is stable!")
else:
    if  IDF1 < 0  or IDF2 < 0: 
        print("Surface is not stable!")        
        # print upper bound for destablising DF value
        if IDF1 <= 0: show(DF1(xi) <= IDF1.upper())
        if IDF2 <= 0: show(DF2(xi) <= IDF2.upper())
    else: print("Cannot determine stability") # "not a > 0" $\color{mycomment}\not\Rightarrow$ "a <= 0" for intervals!
\end{pyin}
\begin{pyprint}
Stability test for test configurations
$\frac{{\left(4 \, \xi - 3\right)} e^{\left(3 \, \xi\right)} + 8 \, \xi}{16 \, \xi^{3}} + \frac{3 \, e^{\left(-\xi\right)}}{16 \, \xi^{3}} \geq \left(-0.0120\right)$
$\frac{{\left(8 \, \xi - 5\right)} e^{\left(3 \, \xi\right)} + 4 \, \xi + 8}{16 \, \xi^{3}} - \frac{3 \, e^{\left(-\xi\right)}}{16 \, \xi^{3}} \geq 0.248$
Surface is not stable!
$\frac{{\left(4 \, \xi - 3\right)} e^{\left(3 \, \xi\right)} + 8 \, \xi}{16 \, \xi^{3}} + \frac{3 \, e^{\left(-\xi\right)}}{16 \, \xi^{3}} \leq \left(-0.00537\right)$
\end{pyprint}

\subsection{Example~\ref{exp:threefold}(No. 2.30)}\ \\
\label{sec:sagemath-code-threefold}
The combinatorial data is is given by $\Box= \text{conv} (\, (-3,0),$ $(-2,1),(2,1),(3,0),(0,-3) \, )$ and  $\Phi_0(x,y) = \min \{0,-x\},\; \Phi_1(x,y)=  \min \{0,y\},\; \Phi_\infty(x,y) = \frac{x-y-1}{2}$.
\begin{pyin}[]
RIF=RealIntervalField(11)
B=Polyhedron(vertices = [[-3,0],[-2,1],[2,1],[3,0],[0,-3]]) # The base polytope
# 3 PL functions on B. If Phi is the minimum of affine functions of the form F(x,y) =
# a + b*x + c*y we represent Phi as the transpose of the matrix with rows [a,b,c].
Phi1 = matrix(QQ,[[0,0,0],[0,-1,0]]).transpose()
Phi2 = matrix(QQ,[[0,0,0],[0,0,1]]).transpose()
Phi3 = matrix(QQ,[[-1/2,1/2,-1/2]]).transpose()
# Form list of degeneration polyhedra
special_fibers = degenerations([Phi1,Phi2,Phi3],B)
\end{pyin}

\textbf{Step (i) -- obtain a closed form for $F_{X,\xi}$}\\
For this we have to analytically solve the integral $\int_\Box \langle u,v \rangle \deg \bar \Phi (u) e^{\langle u, (0,\xi_2) \rangle } du=$\\ $= \int_{\Delta_0} \langle u', (v,0) \rangle \cdot e^{\langle u', (0,\xi_2,0) \rangle} du'$ for $v$ varying over a basis of $N_\mathbb{R}$.

\begin{pyin}[]
L = vector([1,3,13]) # Choose a sufficiently general value for L
# Analytic solution of the integral
F1(xi_2)=intxexp(special_fibers[0],L,vector([0,xi_2,0]),vector([1,0,0])).simplify_full()
F2(xi_2)=intxexp(special_fibers[0],L,vector([0,xi_2,0]),vector([0,1,0])).simplify_full()
show(vector([F1,F2]))
\end{pyin}
\enlargethispage{0.7cm}
\begin{pyprint}
$\xi_{2} \ {\mapsto}\ \left(0,\,\frac{{\left({\left(2 \, \xi_{2}^{3} - 3 \, \xi_{2} - 3\right)} e^{\left(4 \, \xi_{2}\right)} + 12 \, \xi_{2} e^{\left(3 \, \xi_{2}\right)} + 3 \, \xi_{2} + 3\right)} e^{\left(-3 \, \xi_{2}\right)}}{\xi_{2}^{4}}\right)$
\end{pyprint}

\textbf{Step (ii) -- find an estimate for the soliton candidate vector field $\xi$.}
\begin{pyin}[]
# Choose upper and lower bound and hope that the exact solution lies in between
lower = 0.51368; upper = 0.51513
# define a real value xi_RIF between $\color{mycomment} \xi_2^-$ and $\color{mycomment} \xi_2^+$ representing the exact solution.
xi_RIF = RIF(lower,upper)
# Check whether intermediate value theorem guarantees a zero beween
# lower and upper, i.e evaluate F2 at lower and upper using interval arithmetic.
if RIF(F2(lower)) < 0 and RIF(F2(upper)) > 0:
       print "Interval containing solution:", xi_RIF.str(style='brackets')
\end{pyin}
\begin{pyprint}
Interval containing solution: [0.51367 .. 0.51514]
\end{pyprint}
 \textbf{Step (iii) \& (iv) -- obtain closed forms for $\DF_\xi(X_{y,0,1})$ and plug in $\xi$}\\
For this we first have to symbolically solve the integrals  
$h_y(\xi) := \text{vol}(\Delta_y) \cdot \operatorname{DF}_\xi(\mathcal{X}_{y,0,1})=\int_{\Delta_y}  \langle u, (0,0,1) \, \rangle e^{\langle u, (\xi,0) \rangle } du$  
for every (admissible) choice of $y \in \mathbb P^1$ and then plug in the estimate for $\xi$ into the resulting expression.
\begin{pyin}[]
H=[]; Ih=[] # storage for the functions h_y and the values h_y(xi) (as intervals)
# Symbollically solving integrals using Barvinok method
for P in special_fibers:
    h(xi_2) = intxexp(P,L,vector([0,xi_2,0]),vector([0,0,1])).simplify_full()   
    DF=RIF(h(xi_RIF))
    H.append(h); Ih.append(DF)
    show(h(xi_2) >= DF.lower())
if all([h > 0 for h in Ih]): print("Threefold is stable!")
else:
    if any([h <= 0 for h in Ih]):
        print("Threefold is not stable!")
        # print upper bound for destablising DF value
        if Ih[0] <= 0: show(H[0](xi_2) <= Ih[0].upper())
        if Ih[1] <= 0: show(H[1](xi_2) <= Ih[1].upper())
        if Ih[2] <= 0: show(H[2](xi_2) <= Ih[2].upper())
        if Ih[3] <= 0: show(H[3](xi_2) <= Ih[3].upper())
    else: print("Cannot determine stability")
\end{pyin}
\begin{pyprint}
  $\frac{{\left({\left(2 \, \xi_{2}^{3} - 3 \, \xi_{2} - 3\right)} e^{\left(4 \, \xi_{2}\right)} + 3 \, {\left(3 \, \xi_{2}^{2} + 2\right)} e^{\left(3 \, \xi_{2}\right)} - 3 \, \xi_{2} - 3\right)} e^{\left(-3 \, \xi_{2}\right)}}{3 \, \xi_{2}^{4}} \geq 0.927$
$\frac{{\left({\left(8 \, \xi_{2}^{3} + 6 \, \xi_{2}^{2} - 3\right)} e^{\left(4 \, \xi_{2}\right)} - 12 \, {\left(3 \, \xi_{2}^{2} - 3 \, \xi_{2} + 1\right)} e^{\left(3 \, \xi_{2}\right)} + 12 \, \xi_{2} + 15\right)} e^{\left(-3 \, \xi_{2}\right)}}{6 \, \xi_{2}^{4}} \geq 2.07$
$-\frac{{\left(2 \, {\left(2 \, \xi_{2}^{3} - 3 \, \xi_{2} - 3\right)} e^{\left(4 \, \xi_{2}\right)} - 3 \, {\left(3 \, \xi_{2}^{2} - 12 \, \xi_{2} + 2\right)} e^{\left(3 \, \xi_{2}\right)} + 12 \, \xi_{2} + 12\right)} e^{\left(-3 \, \xi_{2}\right)}}{6 \, \xi_{2}^{4}} \geq 0.253$
$\frac{{\left({\left(8 \, \xi_{2}^{3} + 6 \, \xi_{2}^{2} - 3\right)} e^{\left(4 \, \xi_{2}\right)} - 3 \, {\left(3 \, \xi_{2}^{2} - 2\right)} e^{\left(3 \, \xi_{2}\right)} - 6 \, \xi_{2} - 3\right)} e^{\left(-3 \, \xi_{2}\right)}}{6 \, \xi_{2}^{4}} \geq 4.06$
Threefold is stable!
\end{pyprint}
\subsection{Example~\ref{exp:asymetric}(No. 3.23)}\ \\
\label{sec:sagemath-code-asymmetric}
The combinatorial data is is given by $\Box= \text{conv} (\, (-3,0),(-2,1),(1,1),(2,0),(2,-1), (0,3)\, )$ and  $\Phi_0(x,y) = \min \{0,-x\},\; \Phi_1(x,y)=  \min \{0,y\},\; \Phi_\infty(x,y) = \min \{-y, \frac{x-y-1}{2} \}$.

\begin{pyin}[]
RIF=RealIntervalField(40)
# The base polytope
B=Polyhedron(vertices = [[-3,0],[-2,1],[1,1],[2,0],[2,-1],[0,-3]])
# 3 PL functions on B. If phi is the minimum of affine functions of the form 
# F(x,y) = a + b*x + c*y we represent phi as the transpose of the matrix 
# with rows [a,b,c].
Phi1 = matrix(QQ,[[0,0,0],[0,-1,0]]).transpose()
Phi2 = matrix(QQ,[[0,0,0],[0,0,1]]).transpose()
Phi3 = matrix(QQ,[[-1/2,1/2,-1/2],[0,0,-1]]).transpose()
# Form list of degeneration polyhedra
special_fibers = degenerations([Phi1,Phi2,Phi3],B)
\end{pyin}

\textbf{Step (i) -- obtain a closed form for $F_{X,\xi}$}\\
For this we have to analytically solve the integral $\int_\Box \langle u,v \rangle \deg \bar \Phi (u) e^{\langle u, \xi \rangle } du=$\\ $= \int_{\Delta_0} \langle u', (v,0) \rangle \cdot e^{\langle u', (\xi,0) \rangle} du'$ for $v$ varying over a basis of $N_\mathbb{R}$.

\begin{pyin}[]
L = vector([3,13,19]) # Choose a sufficiently general value for L
# Solve the integral analytically using Barvinok's recursion:
v1=vector([1,0,0]); v2=vector([0,1,0])
F1(xi_1,xi_2) = intxexp(special_fibers[0],L,vector([xi_1,xi_2,0]),v1).simplify_full()
F2(xi_1,xi_2) = intxexp(special_fibers[0],L,vector([xi_1,xi_2,0]),v2).simplify_full()
show(F1)
show(F2)
\end{pyin}

\begin{pyprint}
  $(\xi_1,\xi_2) \mapsto $%\nopagebreak%
  $
\medmuskip=0mu
\thinmuskip=0mu
\thickmuskip=0mu
-\frac{
  \begin{aligned} 
      &\scriptstyle \big(4 \xi_{1}^{4} e^{(3 \xi_{1})} -
        \big(2 {({(\xi_{1} - 2)} e^{(4\xi_{1})} - {(\xi_{1} + 1)}
                e^{\xi_{1}} + 3 e^{(3 \xi_{1})})}\xi_{2}^{4} 
              + 4 \xi_{1}^{4} e^{(3 \xi_{1})} +
            (2 {(\xi_{1}^{2} - \xi_{1})} e^{(4\xi_{1})} 
         + \xi_{1} e^{(3\xi_{1})} \\
         &\scriptstyle + {(2 \xi_{1}^{2} + \xi_{1})} e^{\xi_{1}}) \xi_{2}^{3} 
       - 2 {(6 \xi_{1}^{2} e^{(3 \xi_{1})} +
                {(\xi_{1}^{3} - 4 \xi_{1}^{2})} e^{(4
                    \xi_{1})} - {(\xi_{1}^{3} + 2
                    \xi_{1}^{2})} e^{\xi_{1}})} \xi_{2}^{2} \\
               &\scriptstyle - {(3 \xi_{1}^{3} e^{(3\xi_{1})} 
             + 2 {(\xi_{1}^{4}-3\xi_{1}^{3})} e^{(4 \xi_{1})} +
                {(2 \xi_{1}^{4} + 3 \xi_{1}^{3})}
                e^{\xi_{1}})} \xi_{2}\big) e^{(4\xi_{2})} \\
              &\scriptstyle  + 2 (2 \xi_{1}^{3} \xi_{2}^{2}e^{(5 \xi_{1})} 
               + {(3 \xi_{1}^{2} +
                {(2 \xi_{1}^{2} - \xi_{1})} e^{(5\xi_{1})} + \xi_{1})} \xi_{2}^{3} -
            {(3 \xi_{1}^{4} + 3 \xi_{1}^{3} + {(2
                    \xi_{1}^{4} - 3 \xi_{1}^{3})} e^{(5
                    \xi_{1})})} \xi_{2} \\
& \scriptstyle -2{(\xi_{1}^{5} - \xi_{1}^{4})} e^{(5
                \xi_{1})}) e^{(3 \xi_{2})} - 4
        {(\xi_{1}^{3} \xi_{2}^{2} e^{(5 \xi_{1})} -
            {(\xi_{1}^{5} - \xi_{1}^{4})} e^{(5
                \xi_{1})})} e^{(2\xi_{2})}\big) e^{(-3 \xi_{2})}
\end{aligned}}{2   {\left(\xi_{1}^{7} \xi_{2} e^{\left(3   \xi_{1}\right)} - 2   \xi_{1}^{5} \xi_{2}^{3} e^{\left(3   \xi_{1}\right)} + \xi_{1}^{3} \xi_{2}^{5} e^{\left(3   \xi_{1}\right)}\right)}}$

  $(\xi_1,\xi_2) \mapsto $%\nopagebreak%
  $
\medmuskip=0mu
\thinmuskip=0mu
\thickmuskip=0mu
-\frac{
    \begin{aligned}
      &\scriptstyle \big(6 \xi_{1}^{4} \xi_{2} e^{3 \xi_{1}} - 6
      \xi_{1}^{2} \xi_{2}^{3} e^{3 \xi_{1}} + 2 \xi_{1}^{4} e^{3
        \xi_{1}} - 6
      \xi_{1}^{2} \xi_{2}^{2} e^{3 \xi_{1}}-\;\big(\xi_{2}^{5} {\left(2 e^{4 \xi_{1}} - 3 e^{3 \xi_{1}} +
          e^{\xi_{1}}\right)} - 2 \xi_{1}^{4} \xi_{2} e^{3 \xi_{1}} \\
      &\scriptstyle +{\left(2 {\left(\xi_{1} - 1\right)} e^{4 \xi_{1}} -
          {\left(\xi_{1} - 3\right)} e^{3 \xi_{1}}
          - {\left(\xi_{1} + 1\right)} e^{\xi_{1}}\right)}\xi_{2}^{4} + 2 \xi_{1}^{4} e^{3 \xi_{1}}\\
      &\scriptstyle-\; {\left(2 {\left(\xi_{1}^{2} + 2 \xi_{1}\right)} e^{4
            \xi_{1}} - {\left(5 \xi_{1}^{2} + 2 \xi_{1}\right)} e^{3
            \xi_{1}} + {\left(\xi_{1}^{2} - 2 \xi_{1}\right)}
          e^{\xi_{1}}\right)} \xi_{2}^{3} \\
      &\scriptstyle-\; {\left(2 {\left(\xi_{1}^{3} + \xi_{1}^{2}\right)} e^{4
            \xi_{1}} - {\left(\xi_{1}^{3} - 3 \xi_{1}^{2}\right)} e^{3
            \xi_{1}} - {\left(\xi_{1}^{3} - \xi_{1}^{2}\right)}
          e^{\xi_{1}}\right)} \xi_{2}^{2}\big) e^{4 \xi_{2}}\\
      &\scriptstyle +\; 2 {\left(\xi_{1}^{4} e^{5 \xi_{1}} - 3 \xi_{1}^{2}
          \xi_{2}^{2} e^{5 \xi_{1}} - 2 {\left(\xi_{1} e^{5 \xi_{1}}
              - \xi_{1}\right)} \xi_{2}^{3}\right)} e^{3 \xi_{2}}\\
      &\scriptstyle -\; 2 {\left(\xi_{1}^{4} \xi_{2} e^{5 \xi_{1}} - \xi_{1}^{2}
          \xi_{2}^{3} e^{5 \xi_{1}} + \xi_{1}^{4} e^{5 \xi_{1}} - 3
          \xi_{1}^{2} \xi_{2}^{2} e^{5 \xi_{1}}\right)}
      e^{2\xi_{2}}\big)e^{-3 \xi_{2}}
    \end{aligned}
  } {2 {\left(\xi_{1}^{6} \xi_{2}^{2} e^{3 \xi_{1}} - 2 \xi_{1}^{4}
        \xi_{2}^{4} e^{3 \xi_{1}} + \xi_{1}^{2} \xi_{2}^{6} e^{3
          \xi_{1}}\right)}}$

\end{pyprint}
\textbf{Step (ii) -- find an estimate for the soliton candidate vector field $\xi$.}\\
We identify a small closed rectangle containing our estimate such that $ \nabla_n G > 0 $ for any outer normal of this rectangle, where $ G(\xi) = \int_{\Delta_0} e^{\langle u', (\xi,0) \rangle} du'$.
 This and uniqueness guarantee our candidate lies within the rectangle.
\begin{pyin}[]
x0=[ 0.26617786,  0.67164063] # Approximation for solution (e.g. obtained numerically)
#We hope xi lies in the square with centre x0 and side length 2e, where:
e = 0.00001
#We use interval arithmetic to check positivity of $\color{mycomment} \nabla_n G $ on the 
# boundary of the square. We check line segments along the boundary of length 2e/N.
N=2300
# Fix a rectangular region and hope for a solution in there
xi=[RIF(x0[0]-e,x0[0]+e), RIF(x0[1]-e,x0[1]+e)]
if all([RIF(F1(x0[0]+e,RIF(x0[1]-e+i*(2*e/N),x0[1]-e+(i+1)*(2*e/N)))) > 0
        for i in range(N)]):
    if all([RIF(-F1(x0[0]-e,RIF(x0[1]-e+i*(2*e/N),x0[1]-e+(i+1)*(2*e/N)))) > 0 
            for i in range(N)]):
        if all([RIF(F2(RIF(x0[0]-e+i*(2*e/N),x0[0]-e+(i+1)*(2*e/N)), x0[1]+e)) > 0
                for i in range(N)]):
            if all([RIF(-F2(RIF(x0[0]-e+i*(2*e/N),x0[0]-e+(i+1)*(2*e/N)), x0[1]-e)) > 0
                    for i in range(N)]):
                print "There is a root in"
                print xi[0].str(style="brackets"),"x",xi[1].str(style="brackets")
\end{pyin}
\begin{pyprint}
There is a root in
[0.26616785999976 .. 0.26618786000018] x [0.67163062999952 .. 0.67165063000085]
\end{pyprint}
 \textbf{Step (iii) \& (iv) -- obtain closed forms for $\DF_\xi(X_{y,0,1})$ and plug in $\xi$}\\
For this we first have to symbolically solve the integrals  
$h_y(\xi) := \text{vol}(\Delta_y) \cdot \operatorname{DF}_\xi(\mathcal{X}_{y,0,1})=\int_{\Delta_y}  \langle u, (0,0,1) \, \rangle e^{\langle u, (\xi,0) \rangle } du$  
for every choice of $y \in \mathbb P^1$ and then plug in the estimate for $\xi$ into the resulting expression.
\begin{pyin}[]
Ih=[] # Storage for the values h_y(xi) (as intervals)
print "Lower bounds for DF invariants"
for P in special_fibers:
    h(xi_1,xi_2) = intxexp(P,L,vector([xi_1,xi_2,0]),vector([0,0,1])).simplify_full()
    DF=RIF(h(*xi))
    Ih.append(DF)
    print(DF.lower())
if all([h > 0 for h in Ih]): print("Threefold is stable!")
else:
    if any([h <= 0 for h in Ih]): print("Threefold is not stable!")
        print("negative upper bounds for DF invariant:")
        # print upper bound for destablising DF value
        if Ih[0] <= 0: print(Ih[0].upper())
        if Ih[1] <= 0: print(Ih[1].upper())
        if Ih[2] <= 0: print(Ih[2].upper())
        if Ih[3] <= 0: print(Ih[3].upper())
    else: print("Cannot determine stability")
\end{pyin}
\enlargethispage{1cm}
\begin{pyprint}
lower bounds for DF invariants:
1.2766364162
1.8401675971
0.10047917120
3.4443270408
Threefold is stable!  
\end{pyprint}

\bibliographystyle{alpha}
\bibliography{dpsolitons}
\end{document}